\pdfoutput=1
\RequirePackage[l2tabu, orthodox]{nag}
\documentclass{amsart}
\usepackage[T1]{fontenc}
\usepackage[utf8]{inputenc}

        \title[Equivariant Morse theory, Vietoris--Rips complexes, universal spaces]%
              {Equivariant Morse theory on Vietoris--Rips complexes \& universal spaces for proper actions}

       \author{Marco Varisco}
      \address{University at Albany, State University of New York, USA\bigskip}
\email{\hemail{mvarisco@albany.edu}}
\urladdr{\surl{albany.edu/~mv312143}}

       \author{Matthew C.\ B.\ Zaremsky}
\email{\hemail{mzaremsky@albany.edu}}
\urladdr{\surl{albany.edu/~mz498674}}

\makeatletter
\@ifclasslater{amsart}{2020/05/21}{}{%

\@namedef{subjclassname@2020}{%
  \textup{2020} Mathematics Subject Classification}
}
\makeatother


         \date{June 9, 2021}

\usepackage{amssymb}
\usepackage{mathtools}
\usepackage{enumitem}
\usepackage{etoolbox}
\usepackage{aliascnt}
\usepackage{mfirstuc}
\usepackage{tikz}\usetikzlibrary{cd, bending}
\usepackage{microtype}
\usepackage[pagebackref, pdfinfo={
        Title={Equivariant Morse theory on Vietoris--Rips complexes \& universal spaces for proper actions},
       Author={Marco Varisco and Matthew C. B. Zaremsky},
      Subject={MSC2020: 20F65, 20F67, 57M07, 55R35},
}]{hyperref}
\usepackage[alphabetic, backrefs, msc-links, nobysame]{amsrefs}

\newcommand*{\hurl}  [2][www.] {\href{http://#1#2}{\nolinkurl{#2}}}
\newcommand*{\surl}  [2][www.]{\href{https://#1#2}{\nolinkurl{#2}}}
\newcommand*{\hemail}[1]{\href{mailto:#1}{\nolinkurl{#1}}}
\newcommand*{\DOI}   [1]{\href{https://dx.doi.org/#1}{\nolinkurl{#1}}}
\newcommand*{\arXiv} [1]{\href{https://www.arxiv.org/abs/#1}{\nolinkurl{arXiv:#1}}}
\newcommand*{\MSC}   [1]{\href{https://mathscinet.ams.org/msc/msc2020.html?t=#1}{#1}}
\newcommand*{\Zbl}   [1]{Zbl~\href{https://zbmath.org/?q=an\%3A#1}{#1}}

\setlist{leftmargin=*}
\setlist[enumerate]{label=(\roman*)}

\tikzset{/tikz/commutative diagrams/cells={/tikz/font=\everymath\expandafter{\the\everymath\displaystyle}}}

\numberwithin{equation}{section}

\newcommand*{\definetheorem}[3][equation]{%
  \newaliascnt{#2}{#1}
  \newtheorem{#2}[#2]{#3}
  \aliascntresetthe{#2}
  \expandafter\def\csname #2autorefname\endcsname{#3}
}
\makeatletter
\newcommand*{\definetheoremsame}[2][equation]{%
  \definetheorem[#1]{\zap@space#2 \@empty}{\capitalisewords{#2}}
}
\makeatother

\theoremstyle{plain}
\forcsvlist{\definetheoremsame}%
{construction, corollary, equivariant morse lemma, lemma, proposition, theorem}

\theoremstyle{definition}
\forcsvlist{\definetheoremsame}%
{definition, example, question, remark}

\newcommand*{\define}[5]{%
  \ifstrequal{#2}{*}{\expandafter#1\expandafter*}{\expandafter#1}%
  \csname#4#5\endcsname{#3{#5}}
}
\forcsvlist{\define{\DeclareMathOperator}{}{}{}}%
{card, diam, id, ind, res, Stab}

\forcsvlist{\define{\DeclareMathOperator}{*}{}{}}%
{hocolim}

\newcommand*{\MOR}   [4][]{#2\colon#3\TO[#1]#4}
\newcommand*{\TO}    [1][]{\stackrel{#1}{\longrightarrow}}
\newcommand*{\MAPSTO}[1][]{\stackrel{#1}{\longmapsto}}
\newcommand*{\INTO}  [1][\quad]{\xhookrightarrow{#1}}

\DeclarePairedDelimiterX\SET[2]{\{}{\}}{\,#1\;\delimsize\vert\;#2\,}
\DeclarePairedDelimiter\REAL{\lvert}{\rvert}

\DeclareMathOperator*{\smallcoprod}{\textstyle\coprod}

\newcommand*{\N}{\mathbb{N}}
\newcommand*{\Z}{\mathbb{Z}}
\newcommand*{\Q}{\mathbb{Q}}
\newcommand*{\R}{\mathbb{R}}

\newcommand*{\Eu}[1]{\underline{E}#1}

\newcommand*{\rips}{\mathcal{VR}}
\newcommand*{\poset}{\mathcal{P}_{\text{fn}}}

\newcommand*{\da}{{\mkern1.5mu\downarrow}}
\newcommand*{\dst} {\operatorname{st}_\da\!}
\newcommand*{\dlk} {\operatorname{lk}_\da\!}
\newcommand*{\dulk}{\operatorname{lk}_\da^\vee\!}
\newcommand*{\ddlk}{\operatorname{lk}_\da^\wedge\!}
\newcommand*{\dulkposet} {\mathcal{L}_\da^\vee}
\newcommand*{\ddlkposet} {\mathcal{L}_\da^\wedge}

\newif\ifmorseheight
\newif\ifmorseaction
\newif\ifmorsedlinks
\newif\ifmorsedstars
\newif\ifmorselevels
\newif\ifmorsetoplevel
\newif\ifmorsemidlevel
\newif\ifmorseconedoff
\pgfkeys{/morse/.is family, /morse/.cd,
  height/.is if=morseheight,
  action/.is if=morseaction,
  dlinks/.is if=morsedlinks,
  dstars/.is if=morsedstars,
  levels/.is if=morselevels,
toplevel/.is if=morsetoplevel,
midlevel/.is if=morsemidlevel,
conedoff/.is if=morseconedoff,
}
\newcommand{\morse}[1][]{
\pgfkeys{/morse,
  height=true,
  action=true,
  dlinks=true,
  dstars=true,
  levels=true,
toplevel=true,
midlevel=true,
conedoff=true,
#1}

\colorlet{vertex}{black}
\colorlet{dlink}{magenta}
\colorlet{top level}{green}
\colorlet{mid level}{cyan}
\colorlet{low level}{orange}

\tikzset{%
       edge/.style={semithick},
       lkst/.style={very thick},
  back lkst/.style={thick, dotted},
      point/.style={radius=0.055},
fixed point/.style={radius=0.085},
}

\ifmorseaction
\foreach \y in {2, -2.6}
{
\draw[black!66, -{Stealth[bend=1.2, scale=.75]}]
  (-.3,\y+.3) arc [x radius=0.3, y radius=0.075, start angle=180, delta angle=160];
\draw[black!66, -{Stealth[bend=1.2, scale=.75]}, dash pattern=on 8 off 2]
  (+.3,\y+.3) arc [x radius=0.3, y radius=0.075, start angle=000, delta angle=160];
\draw[black!66, densely dotted]
  (0,\y) -- (0,\y+.17) (0,\y+.27) -- (0,\y+.58);
}
\fi

\ifmorselevels
\foreach \x in {-1, 1}
{
\fill[low level!10]
  (\x,-.7) ellipse [x radius=0.3, y radius=0.04];
}
\fill[low level!10]
   (0,-.7) ellipse [x radius=0.7, y radius=0.08];
\ifmorsemidlevel
\foreach \x in {-1, 1}
{
\fill[mid level!10]
  (\x,0)   ellipse [x radius=1.0, y radius=0.11];
}
\fi
\ifmorsetoplevel
\foreach \x in {-1, 1}
{
\fill[top level!10]
  (\x,+.7) ellipse [x radius=0.3, y radius=0.04];
\draw[top level, dotted]
  (\x,+.7) ellipse [x radius=0.3, y radius=0.04, start angle=0, delta angle=180];
}
\fill[top level!10]
   (0,+.7) ellipse [x radius=0.7, y radius=0.08];
\draw[top level, dotted]
   (0,+.7) ellipse [x radius=0.7, y radius=0.08, start angle=0, delta angle=180];
\filldraw[fill=top level!22, draw=top level]
  (0,-2) -- (-2,0) -- (-1.3,+.7)
  arc [start angle=180, delta angle=180, x radius=0.3, y radius=0.04]
  arc [start angle=180, delta angle=180, x radius=0.7, y radius=0.08]
  arc [start angle=180, delta angle=180, x radius=0.3, y radius=0.04]
  -- (2,0) -- cycle;
\fi
\ifmorsemidlevel
\foreach \x in {-1, 1}
{
\draw[mid level, dotted]
  (\x,0)   ellipse [x radius=1.0, y radius=0.11, start angle=0, delta angle=180];
}
\filldraw[fill=mid level!22, draw=mid level]
  (0,-2) -- (-2,0)
  arc [start angle=180, delta angle=180, x radius=1.0, y radius=0.11]
  arc [start angle=180, delta angle=180, x radius=1.0, y radius=0.11]
  -- cycle;
\fi
\foreach \x in {-1, 1}
{
\draw[low level, dotted]
  (\x,-.7) ellipse [x radius=0.3, y radius=0.04, start angle=0, delta angle=180];
}
\draw[low level, dotted]
   (0,-.7) ellipse [x radius=0.7, y radius=0.08, start angle=0, delta angle=180];
\filldraw[fill=low level!22, draw=low level]
  (0,-2) -- (-1.3,-.7)
  arc [start angle=180, delta angle=180, x radius=0.3, y radius=0.04]
  arc [start angle=180, delta angle=180, x radius=0.7, y radius=0.08]
  arc [start angle=180, delta angle=180, x radius=0.3, y radius=0.04]
  -- cycle;
\ifmorseheight
\draw (3.5-.15,+.7) -- (3.5+.15,+.7) node[right] {$r$};
\draw (3.5-.15,0.0) -- (3.5+.15,0.0) node[right] {$s$};
\draw (3.5-.15,-.7) -- (3.5+.15,-.7) node[right] {$t$};
\draw[densely dashed] (3.5-.15-.1,+.7) -- (1.3+.2,+.7);
\draw[densely dashed] (3.5-.15-.1,0.0) -- (2.0+.2,0.0);
\draw[densely dashed] (3.5-.15-.1,-.7) -- (1.3+.2,-.7);
\fi
\fi

\draw[edge] (0,-2) -- (-2,0) -- (0,2) -- (2,0) -- cycle;
\draw[edge] (-1,+1) -- (+1,-1);
\draw[edge] (-1,-1) -- (+1,+1);

\ifmorsedlinks\ifmorsedstars
  \colorlet{vertex}{dlink}
\fi\fi
\foreach \y in {-2, 0, 2}
  \fill[vertex]  (0,\y) circle [fixed point];
\foreach \x/\y in {-2/0, -1/-1, -1/1, 1/-1, 1/1, 2/0}
  \fill[vertex] (\x,\y) circle [point];

\ifmorseconedoff
\filldraw[fill=dlink!22, draw=dlink, lkst]
  (0,0) -- (-.7,-.7)
  arc [start angle=180, delta angle=180, x radius=0.7, y radius=0.08]
  -- cycle;
\draw[dlink, back lkst]
   (0,-.7) ellipse [x radius=0.7, y radius=0.08, start angle=0, delta angle=180];
\draw[dlink, lkst] (-2,0) -- (-1.3,-.7);
\draw[dlink, lkst] (+2,0) -- (+1.3,-.7);
\foreach \x/\y in {-1.3/-.7, +1.3/-.7, -.7/-.7, +.7/-.7, -2/0, +2/0}
{
\fill[dlink] (\x,\y) circle [point];
}
\fill[dlink]  (0,0)  circle [fixed point];
\fi

\ifmorsedlinks
\foreach \x/\y in {-1/1, 0/2, 0/0, 1/1}
{
\draw[dlink, lkst]
  (\x-.4,\y-.4) arc [x radius=0.4, y radius=0.066, start angle=180, delta angle=180];
\foreach \dir in {-1, 1}
{
\fill[dlink] (\x+\dir*.4,\y-.4) circle [point];
}
\ifmorsedstars
\filldraw[draw=dlink, fill=dlink!22, lkst]
  (\x,\y) --
  (\x-.4,\y-.4) arc [x radius=0.4, y radius=0.066, start angle=180, delta angle=180]
  -- cycle;
\fi
\draw[dlink, back lkst]
  (\x+.4,\y-.4) arc [x radius=0.4, y radius=0.066, start angle=000, delta angle=180];
}
\foreach \x/\y/\dir in {-2/0/1, -1/-1/1, 1/-1/-1, 2/0/-1}
{
\fill[dlink] (\x+\dir*.4,\y-.4) circle [point];
\ifmorsedstars
\draw[dlink, lkst] (\x,\y) -- (\x+\dir*.4,\y-.4);
\fi
}
\fi

\ifmorseheight
\draw[-{Stealth}, edge] (3.5,-2.6) -- (3.5,+2.6) node[anchor=north west, inner sep=6] {$\mathbb{R}$};
\ifmorselevels
\draw[-{To}] (0.8,1.6) -- (3.2,1.6) node[above, midway] {$h$};
\else
\draw[-{To}] (2.3,0.0) -- (3.2,0.0) node[above, midway] {$h$};
\fi
\fi
}

\begin{document}

\begin{abstract}
We formalize an equivariant version of Bestvina--Brady discrete Morse theory, and apply it to Vietoris--Rips complexes in order to exhibit finite universal spaces for proper actions for all asymptotically CAT(0) groups.
\end{abstract}


\maketitle
\tableofcontents
\thispagestyle{empty}


\section{Introduction}

Universal spaces for proper actions play an important role in geometric group theory, equivariant homotopy theory, and the Baum--Connes and Farrell--Jones Conjectures in $K$-theory and $L$-theory.
Let us review the definition:
a \emph{universal space for proper actions}
(aka universal space for the family of finite subgroups)
of a discrete group~$G$ is a proper $G$-CW complex~$\Eu{G}$ such that, for each finite subgroup~$H\le G$, the $H$-fixed point set~$(\Eu{G})^H$ is contractible.
Recall that a $G$-CW complex is a CW complex with a $G$-action such that, for all open cells~$e$ and all~$g\in G$, if $g\cdot e\cap e\neq \emptyset$ then $g\cdot x=x$ for all~$x\in e$.
A $G$-CW complex is proper if and only if all point stabilizers are finite.
For any group~$G$, universal spaces for proper actions exist and are unique up to $G$-homotopy, and they satisfy the following universal property: for any proper $G$-CW complex~$X$ there is up to $G$-homotopy exactly one $G$-map $X\TO\Eu{G}$.
If~$G$ is torsion-free then by definition $\Eu{G}$ is a contractible free $G$-space, i.e., a universal cover of a classifying space~$BG=K(G,1)$.
For more information on universal spaces we refer to Lück's survey article~\cite{L}.

A central question is whether a given group~$G$ has a universal space for proper actions that is finite.
Recall that a $G$-CW complex is called finite if it is cocompact, or, equivalently, if it only has finitely many $G$-orbits of cells.
If $G$ is torsion-free, having a finite~$\Eu{G}$ means having a finite~$BG$, i.e., being of type~$F$; for this reason groups having a finite~$\Eu{G}$ are sometimes said to be of type $\underline{F}$ in the literature.

Meintrup and Schick~\cite{meintrup02} proved that hyperbolic groups have finite universal spaces for proper actions, as claimed in~\cite{baum94}*{Section~2}.
This is also known to be true for CAT(0) groups---it follows from basic properties of CAT(0) spaces and a result of Ontaneda~\cite{O}, as we explain below.
A natural simultaneous generalization of hyperbolic and CAT(0) groups is given by asymptotically CAT(0) groups, which were introduced and studied by Kar~\citelist{\cite{kar08} \cite{kar11}} and are reviewed at the end of this introduction.
The class of asymptotically CAT(0) groups contains all hyperbolic groups and all CAT(0) groups, and is closed under taking finite products, amalgamated free products over finite subgroups, HNN extensions along finite subgroups, and relatively hyperbolic overgroups.
Our main result is the following.

\begin{theorem}
\label{main-intro}
All asymptotically CAT(0) groups have finite universal spaces for proper actions.
\end{theorem}

More precisely, in \autoref{main-precise} we show that if $G$ acts properly and cocompactly by isometries on an asymptotically CAT(0) geodesic metric space~$X$, then the Vietoris--Rips complex~$\rips_t(G\cdot x_0)$ of any $G$-orbit in~$X$ is a finite universal space for proper actions for any sufficiently large~$t>0$.
We consider $G\cdot x_0$ with the induced metric from~$X$.
Incidentally, by the Schwarz--Milnor Lemma this metric is quasi-isometric to the word metric on~$G$, but this is not enough to conclude that~$\rips_t(G)$ is also an~$\Eu{G}$; see \autoref{Rips(G)} for more in this vein.

The proof of our main result relies on an equivariant version of Bestvina--Brady discrete Morse theory, which we develop in \autoref{MORSE} and believe to be of independent interest, and which we then apply to Vietoris--Rips complexes in \autoref{RIPS}.
Since its introduction in~\cite{bestvina97}, Bestvina--Brady discrete Morse theory has proven to be an essential tool in a variety of applications for deducing ``global'' topological properties from ``local'' topological information.
Even though the equivariant version we discuss here relies on the same constructions as in~\cite{bestvina97}, some care is needed in formulating the appropriate equivariant definitions and statements, and in checking the equivariance of the constructions.

For hyperbolic groups, \autoref{main-intro} was proved by Meintrup and Schick~\cite{meintrup02}, as we already mentioned.
Their argument also uses Vietoris--Rips complexes, but does not use discrete Morse theory, so in the hyperbolic case our approach can be viewed as an independent Morse-theoretic proof of their result.

In the CAT(0) case, if $G$ acts properly and cocompactly by isometries on a CAT(0) geodesic metric space~$X$ in such a way that $X$ has the structure of a $G$-CW complex, then $X$ itself is a finite~$\Eu{G}$.
This is well known and follows at once from the fact~\cite{BH}*{Corollary~II.2.8, page~179} that any finite group acting on a complete CAT(0) space has nonempty and convex fixed point set, and completeness follows from the assumptions on the action by~\cite{BH}*{I.8.4(1), page 132}.
If we drop the assumption about the $G$-CW structure on~$X$, then \cite{O}*{Proposition~A, page~48} shows that $X$ is $G$-homotopy equivalent to a finite $G$-CW complex, which is consequently a finite~$\Eu{G}$.
To the best of our knowledge, however, the fact that Vietoris--Rips complexes provide universal spaces for proper actions for CAT(0) groups seems to be new.

An immediate consequence of \autoref{main-intro} is that asymptotically CAT(0) groups have only finitely many conjugacy classes of finite subgroups---a result that was already proved by Kar~\cite{kar11}*{Theorem~10}.

Using a theorem of Riley~\cite{riley03}, Kar showed in~\cite{kar11}*{Proposition~13} that all asymptotically CAT(0) groups are of type~$F_\infty$, i.e., have a classifying space with only finitely many cells in each dimension.
This result was improved by the second author, who in~\cite{zaremsky18}*{Theorem~6.2} showed that all asymptotically CAT(0) groups are of type~$F_*$, i.e., act properly and cocompactly on a contractible complex.
Our approach here relies heavily on the ideas from~\cite{zaremsky18}, but is completely independent from the results proved there.
Our \autoref{link-criterion} and \autoref{main-precise} can be thought of as equivariant generalizations of~\cite{zaremsky18}*{Theorems~3.5 and~6.2}.

\medskip

We conclude this introduction by reviewing asymptotically CAT(0) spaces and groups, and a select few applications of our main result.
The idea behind the definition of asymptotically CAT(0) is to use a comparison triangle inequality, just like in the definition of CAT(0), but to allow an error term that grows sublinearly with the size of the triangles.

More precisely, let $X$ be a geodesic metric space.
Let $\Delta$ be a geodesic triangle in~$X$ with vertices $x_1$, $x_2$, and~$x_3$.
The corresponding \relax{comparison triangle}~$\overline{\Delta}$ is a geodesic triangle in~$\R^2$ with vertices $\overline{x_1}$, $\overline{x_2}$, and $\overline{x_3}$, such that $d(x_i,x_j)=d(\overline{x_i},\overline{x_j})$ for all $i,j\in\{1,2,3\}$.
Given a point $p\in \Delta$ on the geodesic from $x_i$ to $x_j$, the corresponding \relax{comparison point} in~$\overline{\Delta}$ is the point~$\overline{p}$ on the geodesic from $\overline{x}_i$ to~$\overline{x}_j$ with $d(\overline{x_i},\overline{p})=d(x_i,p)$.
Set
\(
d(\Delta)=\max\SET{d(x_i,x_j)}{1\le i< j\le 3}
\).

\begin{definition}[asymptotically CAT(0) spaces and groups]
\label{asymp-CAT(0)}
A geodesic metric space~$X$ is \emph{asymptotically CAT(0)} if there exists a function $\MOR{\sigma}{\R}{\R}$ that is sublinear, i.e.,  $\lim_{t\to\infty}\frac{\sigma(t)}{t}=0$, and such that for any geodesic triangle~$\Delta$ in~$X$ and any points $p,q\in \Delta$ we have
\[
d(p,q)\le d(\overline{p},\overline{q})+\sigma(d(\Delta))
\,.
\]
An \emph{asymptotically CAT(0) group} is a group that admits a proper and cocompact action by isometries on an asymptotically CAT(0) geodesic metric space.
\end{definition}

If we require $\sigma(t)=0$ for all~$t$ then we recover the definition of CAT(0).

Kar defined a metric space to be asymptotically CAT(0) if all of its asymptotic cones are CAT(0), and then showed in~\cite{kar11}*{Theorem~8} that for geodesic metric spaces this quicker definition is equivalent to the more hands-on \autoref{asymp-CAT(0)} above.

\begin{remark}[invariance under quasi-isometries]
\label{qi}
As with being CAT(0), being asymptotically CAT(0) is not a quasi-isometry invariant of spaces.
For example, $\Z^2$ with the taxi-cab distance (i.e., the word metric with respect to the standard generating set) is not asymptotically CAT(0), but it is quasi-isometric to the CAT(0) euclidean plane~$\R^2$.
Moreover, being CAT(0) is not a quasi-isometry invariant of groups; see, e.g., \cite{piggott10}*{Remark~1}.
To the best of our knowledge, it is an open question whether the non-CAT(0) examples in loc.\ cit.\ are asymptotically CAT(0), and, more generally, whether being asymptotically CAT(0) is a quasi-isometry invariant of groups.
Also the question of whether being CAT(0) or asymptotically CAT(0) is a commensurability invariant of groups seems to be open.
\end{remark}

As mentioned earlier, Kar proved that the class of asymptotically CAT(0) groups is closed under the following operations:
\begin{itemize}
\item
if $G$ and~$H$ are asymptotically CAT(0), then so is their product~$G\times H$~\cite{kar08}*{Proposition~24};
\item
if $G$ and~$H$ are asymptotically CAT(0), and $F$ is a finite group together with monomorphisms $F\to G$ and~$F\to H$, then the amalgamated free product $G *_F H$ is asymptotically CAT(0)~\cite{kar11}*{Theorem~19};
\item
if $G$ is asymptotically CAT(0), and $F$ is a finite group together with monomorphisms $F\rightrightarrows G$, then the HNN extension $G *_F$ is asymptotically CAT(0)~\cite{kar11}*{Theorem~19};
\item
if $G$ is relatively hyperbolic with respect to a subgroup~$H$, and $H$ is asymptotically CAT(0), then so is~$G$~\cite{kar11}*{Theorem~20}.
\end{itemize}
From this one gets many examples of groups for which \autoref{main-intro} was not previously known.

\medskip

Finally, we highlight a few known consequences of our main result.
Assume that a group~$G$ has a finite universal space for proper actions, which is true if $G$ is asymptotically CAT(0) thanks to \autoref{main-intro}.
Then:
\begin{itemize}
\item
\emph{(Equivariant stable homotopy theory)}
The $G$-sphere spectrum is a small object in the proper equivariant stable homotopy category of orthogonal $G$-spectra
by~\citelist{\cite{DHLPS}*{Proposition~2.1.4} \cite{BDP}*{Theorems 5.1 and~5.4}}.
\item
\emph{(Algebraic $K$-theory)}
The rationalized Farrell--Jones assembly map for the algebraic $K$-theory groups~$K_n(\Z[G])\otimes_\Z \Q$ of the integral group ring of~$G$ is eventually injective, i.e., injective for all sufficiently large~$n$,
by~\cite{LRRV}*{Theorem~1.15}.
\end{itemize}
If~$G$ not only has a finite~$\Eu{G}$, but this~$\Eu{G}$ also has a metrizable compactification satisfying further technical conditions, then the assembly maps with respect to the family of finite subgroups for $K_n(R[G])$ and $L_n(R[G])$ are split injective by~\citelist{\cite{R-K} \cite{R-L}}.
Building upon~\cite{meintrup02} (i.e., \autoref{main-intro} for hyperbolic groups), Rosenthal and Schütz~\cite{RS} showed that hyperbolic groups satisfy all these technical conditions, and hence the results of~\citelist{\cite{R-K} \cite{R-L}} apply.
It is natural to ask whether the same is true for all asymptotically CAT(0) groups and, more generally, whether asymptotically CAT(0) groups satisfy the Novikov, Borel, Baum--Connes, and Farrell--Jones Conjectures.
This is already known to be the case for hyperbolic and CAT(0) groups.
(For an introduction to these conjectures, see for example~\cite{RV} and the references therein.)


\subsection*{Acknowledgments}
We are grateful to the referee for their careful and thoughtful report and for their helpful suggestions, which made this a better article.
We also thank Mladen Bestvina, Nima Hoda, Aditi Kar, Wolfgang Lück, Michał Marcinkowski, and David Rosenthal for helpful conversations.
This work was supported by grants from the Simons Foundation (\#419561, Marco Varisco, and \#635763, Matthew Zaremsky).


\section{Equivariant discrete Morse theory}
\label{MORSE}

In this section, we introduce an equivariant version of the main definitions and results of Bestvina--Brady discrete Morse theory, which we ultimately use to formulate and prove the following result.

\begin{proposition}
\label{morse-useful}
Let $Y$ be an affine $G$-CW complex together with a $G$-Morse function~$\MOR{h}{Y}{\R}$,
and $t,s\in \R\cup \{\infty\}$ with~$t<s$.
Assume that for every $y\in Y^{(0)}\cap h^{-1}\bigl((t,s]\bigr)$ the descending link $\dlk y$ is $\Stab_G(y)$-contractible.
Then the inclusion $Y^{h\le t}\INTO Y^{h\le s}$ is a $G$-homotopy equivalence.
\end{proposition}

For the applications to Vietoris--Rips complexes in the next section, this result is the only Morse-theoretic input that we need.
The experts on Bestvina--Brady discrete Morse theory may wish to skip the remainder of this section---the takeaway is that the main arguments of \cite{bestvina97}*{in particular Lemmas~2.3 and~2.5} can be carried through equivariantly.
(The level of generality in this section is also greater than what is needed in the next. However, the authors believe that the benefits of following~\cite{bestvina97} outweigh the costs: essentially, only \autoref{lexi} could be avoided by adopting a Morse-theoretic setup more tailored to Vietoris--Rips complexes.)

For the nonexperts, we proceed to spell out all details---not only for the sake of a complete and self-contained exposition, but also because we found that checking the equivariance of the constructions required us to make explicit some subtle details that are only implicit in~\cite{bestvina97}.

We formulate all definitions in the presence of an action of a discrete group $G$.
We do it in such a way that, when the action (or the group) is trivial, we recover the usual nonequivariant definitions from~\cite{bestvina97}.

\begin{definition}[affine $G$-CW complexes]
\label{affineGCW}
An \emph{affine $G$-CW complex} is a $G$-CW complex~$Y$ with a chosen CW structure satisfying the following conditions:
\begin{enumerate}[label=(\arabic*), start=0]
\item
$Y$ is a regular CW complex, i.e., characteristic maps are homeomorphisms.
\item
Each (closed) cell~$e$ of~$Y$ has a chosen characteristic map $\MOR{\chi_e}{C_e}{e}$ with $C_e$ a convex polyhedron in some euclidean space~$\R^{n_e}$.
An \emph{admissible} characteristic map for~$e$ is any composite $C\TO[\alpha] C_e\TO[\chi_e] e$, where $C$ is a convex polyhedron in some~$\R^{m}$ and $\alpha$ is an affine homeomorphism, i.e., a homeomorphism $C\TO C_e$ induced by restricting an affine map~$\R^{m}\TO \R^{n_e}$.
\item
For each cell~$e$, the restriction of~$\chi_e$ to any face of~$C_e$ is an admissible characteristic map for some face~$e'$ of~$e$.
\item
For each cell~$e$ and each~$g\in G$, $g\circ\chi_e$ is an admissible characteristic map for~$g\cdot e$.
\end{enumerate}
\end{definition}

A remark is in order about the regularity assumption in our definition of affine CW complexes.
Regularity is not assumed in~\cite{bestvina97}, but it is assumed in the survey~\cite{bestvina08}.
In what follows regularity is not inherently necessary, but it is convenient, and in practice every concrete application of which we are aware (including the main one in~\cite{bestvina97}) involves regular CW complexes, such as simplicial or cubical complexes.
So, for us the slight loss of generality is balanced out by the comparative ease afforded by assuming regularity.
Also, condition~(2) above imposes some amount of regularity anyway, as does the existence of a $G$-Morse function (see next definition); for example, an affine $G$-CW complex with a $G$-Morse function necessarily has regular 1-skeleton.
In any case, one can always assume regularity by passing to a subdivision, which would usually carry an induced $G$-Morse function.

\begin{definition}[$G$-Morse functions]
Given an affine $G$-CW complex~$Y$, a \emph{$G$-Morse function} on~$Y$ is a $G$-equivariant map $\MOR{h}{Y}{\R}$, where $G$ acts trivially on~$\R$, satisfying the following conditions:
\begin{enumerate}[label=(\arabic*)]
\item
$h(Y^{(0)})$ is closed and discrete in~$\R$.
\item
$h$ is nonconstant on edges.
\item
$h$ is affine on cells, meaning for each cell~$e$ we have that $\MOR{h\circ \chi_e}{C_e}{\R}$ is the restriction to~$C_e$ of an affine map~$\R^{n_e}\TO\R$.
\end{enumerate}
\end{definition}

For the rest of this section, let $Y$ be an affine $G$-CW complex together with a $G$-Morse function~$h$.
Note that for each~$e$ the function~$h$ achieves its maximum value on~$e$ at a unique vertex.
This is because~$h$ is affine on cells and nonconstant on edges.
We call this the \emph{top} vertex of~$e$.
This observation leads us to the following definition.

\begin{definition}[descending links \& stars]
\label{dlk-dst}
If~$Y$ is an affine $G$-CW complex together with a $G$-Morse function~$\MOR{h}{Y}{\R}$, and~$y\in Y^{(0)}$ is a vertex, the \emph{descending link}~$\dlk y$ is the $\Stab_G(y)$-CW complex defined as follows.
Since~$h$ is discrete on vertices we can choose~$t<h(y)$ such that no vertices of~$Y$ have $h$-value in~$\bigl(t,h(y)\bigr)$.
Now define~$\dlk y$ to be the subspace of~$Y$ that is the union of all cells with~$y$ as their top vertex, intersected with~$h^{-1}(t)$.
More precisely, we may refer to this subspace as the copy of~$\dlk y$ \emph{at level~$t$}, but up to homeomorphism it is independent of the choice of~$t$ with~$Y^{(0)}\cap h^{-1}\bigl(\bigl(t,h(y)\bigr)\bigr)= \emptyset$.
Note that~$\dlk y$ inherits a CW structure from~$Y$: for any $(k+1)$-cell~$e$ with top vertex~$y$, the preimage~$\chi_e^{-1}\bigl(e\cap h^{-1}(t)\bigr)\subseteq C_e$ is a $k$-cell in some euclidean space, and restricting~$\chi_e$ to this cell gives a characteristic map to the corresponding cell in~$\dlk y$.
The \emph{descending star}~$\dst y$ is constructed analogously to~$\dlk y$, but is the union of all cells with~$y$ as their top vertex, intersected with~$h^{-1}\bigl(\bigl[t,h(y)\bigr]\bigr)$.
Note that~$\dst y$ is a cone on~$\dlk y$, and has a CW structure constructed analogously.
Finally, note that~$\Stab_G(y)$ evidently acts on $\dlk y$ and~$\dst y$.
With respect to these actions $\dlk y$ and~$\dst y$ are $\Stab_G(y)$-CW complexes, and the inclusions into~$Y$ give $\Stab_G(y)$-equivariant embeddings
$\dlk y\INTO \res^G_{\Stab_G(y)} Y$
and
$\dst y\INTO \res^G_{\Stab_G(y)} Y$,
where $\res^G_{\Stab_G(y)} Y$ denotes the space~$Y$ considered as a $\Stab_G(y)$-space by restriction of the $G$-action.
\end{definition}

\begin{example}[hollow quilt]
\label{example}
The following pictures show an example of an affine $(\Z/2\Z)$-CW complex with a $(\Z/2\Z)$-Morse function.
Each diamond in the picture is a ``pocket'' of two squares glued (``sewn'') along their boundaries, so the $(\Z/2\Z)$-action given by rotating about the vertical axis is free on all points but the vertices marked by larger dots.
The $(\Z/2\Z)$-Morse function is given by the vertical height.
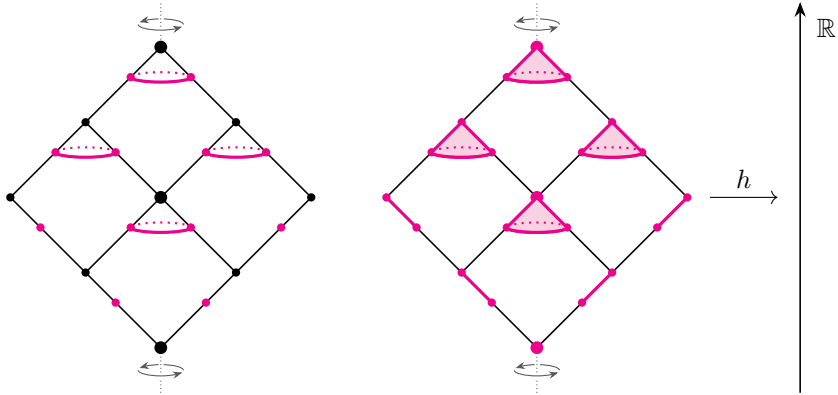
\[
\begin{tikzpicture}
\begin{scope}[xshift=-5.0cm]
\morse[levels=false, conedoff=false, dstars=false, height=false]
\end{scope}
\morse[levels=false, conedoff=false]
\end{tikzpicture}
\]
The pictures on the left and on the right indicate in color the descending link and the descending star, respectively, of each vertex (at appropriate levels~$t$).
The vertices marked by larger dots have nontrivial stabilizer, while the vertices marked by smaller dots have trivial stabilizer.
For each vertex $y$, the spaces $\dlk y$ and $\dst y$ are visibly $\Stab_{\Z/2\Z}(y)$-spaces.
\end{example}

The next lemma says that the description of descending links in \autoref{dlk-dst} simplifies when~$Y$ is a simplicial complex, and even more so when~$Y$ is a flag complex, as in the case of Vietoris--Rips complexes in \autoref{RIPS}.

\begin{lemma}
\label{dlk-in-simplicial-cx}
Let~$Y$ be an affine $G$-CW complex together with a $G$-Morse function~$\MOR{h}{Y}{\R}$, and let~$y$ be a vertex.
If~$Y$ is a simplicial complex, then~$\dlk y$ is homeomorphic to the subcomplex of~$Y$ consisting of all simplices~$d$ for which $y\not\in d$ but~$d\cup\{y\}\subseteq e$ for some simplex~$e$ with top vertex~$y$.
If moreover~$Y$ is a flag complex, then~$\dlk y$ is homeomorphic to the full subcomplex of~$Y$ spanned by vertices~$y'$ such that~$y'$ is adjacent to~$y$ and~$h(y')<h(y)$.
\end{lemma}

\begin{proof}
Assume that~$Y$ is simplicial.
First note that the cells of~$\dlk y$ are simplices, since for any $(k+1)$-simplex~$e$ of $Y$ with top vertex~$y$, the preimage~$\chi_e^{-1}\bigl(e\cap h^{-1}(t)\bigr)\subseteq C_e$ is a $k$-simplex.
We claim that $\dlk y$ is a simplicial complex, i.e., that every simplex is uniquely determined by its vertex set.
Given simplices~$e,e'$ in~$Y$ with top vertex~$y$, if~$e$ and~$e'$ have the same set of edges with top vertex~$y$ then they consequently have the same vertex set, which then implies that~$e=e'$.
In particular, since the vertices of~$\dlk y$ (say the copy at some level~$t$) are the points with $h$-value~$t$ in the edges of~$Y$ with top vertex~$y$, if two simplices in~$\dlk y$ have the same vertex set then we conclude that they are equal.
This shows that~$\dlk y$ is a simplicial complex, as claimed.
Now it suffices to show that~$\dlk y$ is isomorphic as a simplicial complex to the desired subcomplex of~$Y$.
A $k$-simplex of~$\dlk y$ is~$h^{-1}(t)\cap e$ for some $(k+1)$-cell~$e$ of~$Y$ with top vertex~$y$, and we declare that this $k$-simplex of~$\dlk y$ maps to the $k$-simplex of~$Y$ that is the maximal face of~$e$ not containing~$y$.
This clearly yields a simplicial isomorphism from~$\dlk y$ to the desired subcomplex.
Finally, if~$Y$ is a flag complex then the last result is immediate from the flag condition.
\end{proof}

Given $t\in \R$, set
\[
Y^{h\le t}=h^{-1}\bigl((-\infty,t]\bigr)
\qquad\text{and}\qquad
Y^{h\le \infty}=Y
=
\bigcup_{t\in \R} Y^{h\le t}
.
\]
Since $h$ is $G$-equivariant, each $Y^{h\le t}$ is $G$-invariant.

The main result of equivariant discrete Morse theory, the \autoref{morse-lemma}, is conveniently formulated in terms of a commutative diagram that we construct next.

\begin{construction}
\label{commutative}
Let $Y$ be an affine $G$-CW complex together with a $G$-Morse function~$\MOR{h}{Y}{\R}$,
and $t<s$ in~$\R$ with $Y^{(0)}\cap h^{-1}\bigl((t,s)\bigr)=\emptyset$.
Then~$G$ acts on the disjoint unions (as opposed to unions as subspaces of~$Y$)
\[
\smallcoprod_{y\in Y^{(0)}\cap h^{-1}(s)} \dlk y
\qquad\text{and}\qquad
\smallcoprod_{y\in Y^{(0)}\cap h^{-1}(s)} \dst y
\]
and there is the following commutative diagram of $G$-spaces and $G$-maps.
\begin{equation}
\label{cd:commutative}
\begin{tikzcd}
\smallcoprod_{y\in Y^{(0)}\cap h^{-1}(s)} \dlk y
\arrow[r, "\Phi"]
\arrow[d, hookrightarrow]
&
Y^{h\le t}
\arrow[d, hookrightarrow]
\\
\smallcoprod_{y\in Y^{(0)}\cap h^{-1}(s)} \dst y
\arrow[r, "\Psi"']
&
Y^{h\le s}
\end{tikzcd}
\end{equation}
The left-hand vertical map in~\eqref{cd:commutative} is a $G$-cofibration.
The image of~$\Phi$ in~$Y^{h\le t}$ is the intersection of~$h^{-1}(t)$ with the union of all cells~$e$ whose top vertex~$y$ satisfies~$h(y)=s$, and the image of~$\Psi$ in~$Y^{h\le s}$ is the intersection of~$h^{-1}\bigl([t,s]\bigr)$ with this union.
\end{construction}

\begin{proof}
For each vertex~$y\in Y^{(0)}$ with~$h(y)=s$, since we are assuming that no vertices of~$Y$ have $h$-values in~$(t,s)$, we can take the copies of $\dlk y$ and~$\dst y$ at this fixed level~$t$.
Then, by \autoref{dlk-dst}, we have the following commutative diagram of $\Stab_G(y)$-spaces and $\Stab_G(y)$-maps.
\begin{equation}
\label{cd:commutative-res}
\begin{tikzcd}
\dlk y
\arrow[r, hookrightarrow]
\arrow[d, hookrightarrow]
&
\res^G_{\Stab_G(y)} Y^{h\le t}
\arrow[d, hookrightarrow]
\\
\dst y
\arrow[r, hookrightarrow]
&
\res^G_{\Stab_G(y)} Y^{h\le s}
\end{tikzcd}
\end{equation}

Next, recall that if $H$ is a subgroup of~$G$, the restriction functor~$\res^G_H$ from $G$-spaces to $H$-spaces has a left adjoint functor, induction, which is constructed as follows.
Given a left $H$-space~$X$,
the induced left $G$-space
\(
\ind_H^G X
=
G\times_H X
\)
is defined as the quotient of~$G\times X$ by the diagonal $H$-action given by the right action of~$H$ by multiplication on~$G$ and the left action of~$H$ on~$X$.
The left action of~$G$ by multiplication on itself commutes with the right $H$-action on~$G$ we quotiented out and so it induces a left $G$-action on~$\ind_H^G X$.

Applying induction to~\eqref{cd:commutative-res} we then obtain the following commutative diagram of $G$-spaces and $G$-maps.
\begin{equation}
\label{cd:commutative-ind}
\begin{tikzcd}
\ind^G_{\Stab_G(y)} \dlk y
\arrow[r, "\varphi"]
\arrow[d, hookrightarrow]
&
Y^{h\le t}
\arrow[d, hookrightarrow]
\\
\ind^G_{\Stab_G(y)} \dst y
\arrow[r, "\psi"']
&
Y^{h\le s}
\end{tikzcd}
\end{equation}
Notice that, forgetting the $G$-action, there is a homeomorphism
\begin{equation}
\label{eq:induction-one}
\ind^G_{\Stab_G(y)} \dlk y
\ \cong
\smallcoprod_{y'\in G\cdot y} \dlk y'
\,,
\end{equation}
and left multiplication by~$g$ induces a homeomorphism $\dlk y\TO \dlk(g\cdot y)$ that intertwines the action by~$\Stab_G(y)$ on the source and the  action by~$g\Stab_G(y)g^{-1}=\Stab_G(g\cdot y)$ on the target.
We use the homeomorphism~\eqref{eq:induction-one} to define a $G$-action on the right-hand side of~\eqref{eq:induction-one}.
Exactly the same applies when we replace $\dlk$ with~$\dst$.

Finally, diagram~\eqref{cd:commutative} is easily obtained from~\eqref{cd:commutative-ind} as follows.
Decompose the $G$-set~$Y^{(0)}\cap h^{-1}(s)$ into transitive components, say indexed by a set~$\mathcal{J}$, and choose a point~$y_j$ in each component, so that
\[
Y^{(0)}\cap h^{-1}(s)
=
\smallcoprod_{j\in \mathcal{J}} G\cdot y_j
\cong
\smallcoprod_{j\in \mathcal{J}} G/\Stab_G(y_j)
\,.
\]
We then get homeomorphisms
\begin{equation}
\label{eq:induction-many}
\smallcoprod_{y\in Y^{(0)}\cap h^{-1}(s)} \dlk y
\ \cong
\adjustlimits
\smallcoprod_{j\in\mathcal{J}\ }
\smallcoprod_{y\in Y^{(0)}\cap h^{-1}(s)\cap G\cdot y_j} \dlk y
\ \cong
\smallcoprod_{j\in\mathcal{J}} \ind^G_{\Stab_G(y_j)} \dlk y_j
\,,
\end{equation}
where in the last homeomorphism we use~\eqref{eq:induction-one}.
We use the homeomorphism~\eqref{eq:induction-many} to define a $G$-action on the left-hand side of~\eqref{eq:induction-many}, since the right-hand side is a disjoint union of $G$-spaces.
Now, on each summand in the right-hand side of~\eqref{eq:induction-many} there is a $G$-map~$\varphi_j$ to~$Y^{h\le t}$ as in~\eqref{cd:commutative-ind}, and these maps yield the $G$-map~$\Phi$ in~\eqref{cd:commutative}.
A completely analogous description works also for $\smallcoprod_{y\in Y^{(0)}\cap h^{-1}(s)} \dst y$, completing the construction of the commutative diagram~\eqref{cd:commutative} of $G$-spaces and $G$-maps.

Since each $\dlk y_j\INTO \dst y_j$ is the inclusion of a $\Stab_G(y_j)$-CW subcomplex, it follows from~\eqref{eq:induction-many} that the left-hand vertical map in~\eqref{cd:commutative} is a $G$-cofibration.

The statements about the images of $\Phi$ and~$\Psi$ are true by construction.
\end{proof}

\begin{equivariantmorselemma}
\label{morse-lemma}
Let $Y$ be an affine $G$-CW complex together with a $G$-Morse function~$\MOR{h}{Y}{\R}$,
and $t<s$ in~$\R$ with $Y^{(0)}\cap h^{-1}\bigl((t,s)\bigr)=\emptyset$.
Let $P$ be the pushout of the following diagram.
\begin{equation}
\label{cd:pushout}
\begin{tikzcd}
\smallcoprod_{y\in Y^{(0)}\cap h^{-1}(s)} \dlk y
\arrow[r, "\Phi"]
\arrow[d, hookrightarrow]
&
Y^{h\le t}
\\
\smallcoprod_{y\in Y^{(0)}\cap h^{-1}(s)} \dst y
\end{tikzcd}
\end{equation}
Then the natural map $P\TO Y^{h\le s}$ induced by the commutative diagram~\eqref{cd:commutative} is a $G$-homotopy equivalence.
\end{equivariantmorselemma}

Note that the last statement in \autoref{commutative} implies that $P$ is homeomorphic to the subspace of~$Y^{h\le s}$ given by
\(
Y^{h\le t}\cup \bigcup \SET{e}{e\text{ is a cell with }e\subseteq Y^{h\le s}}
\),
and under this identification the natural map $P\TO Y^{h\le s}$ is the inclusion.

The conclusion of this lemma can be reformulated in more homotopy-theoretic language as follows.
The pushout~$P$ is the $G$-homotopy pushout of diagram~\eqref{cd:pushout},
because the left-hand vertical map in~\eqref{cd:pushout} is a $G$-cofibration.
Then saying that $P\TO Y^{h\le s}$ is a $G$-homotopy equivalence is by definition the same as saying that diagram~\eqref{cd:commutative} is $G$-homotopy cocartesian.

We point out that, even in the nonequivariant case, our formulation of the Morse Lemma combines two statements that are usually given separately:
\begin{enumerate}
\item
If $Y^{(0)}\cap h^{-1}(s)=\emptyset$, then the spaces on the left in~\eqref{cd:pushout} are empty, so the pushout~$P$ can be identified with~$Y^{h\le t}$, and the conclusion in this case is that the inclusion $Y^{h\le t}\INTO Y^{h\le s}$ is a homotopy equivalence; this corresponds to~\cite{bestvina97}*{Lemma~2.3}.
\item
If $Y^{(0)}\cap h^{-1}(s)\not=\emptyset$, then saying that $P\TO Y^{h\le s}$ is a homotopy equivalence is a precise way of formulating~\cite{bestvina97}*{Lemma~2.5}, which states that~$Y^{h\le s}$ is homotopy equivalent to~$Y^{h\le t}$ with the descending links~$\dlk y$ coned off for each vertex~$y$ with~$h(y)=s$ (notice that~$\dst y$ is the cone on~$\dlk y$).
The proof of \cite{bestvina97}*{Lemma~2.5} makes clear that it is the copy of~$\dlk y$ at level~$t$ that is being coned off.
\end{enumerate}

\begin{example}
The following pictures show an example of the \autoref{morse-lemma}, applied to the affine $(\Z/2\Z)$-CW complex~$Y$ (the ``hollow quilt'') from \autoref{example}.
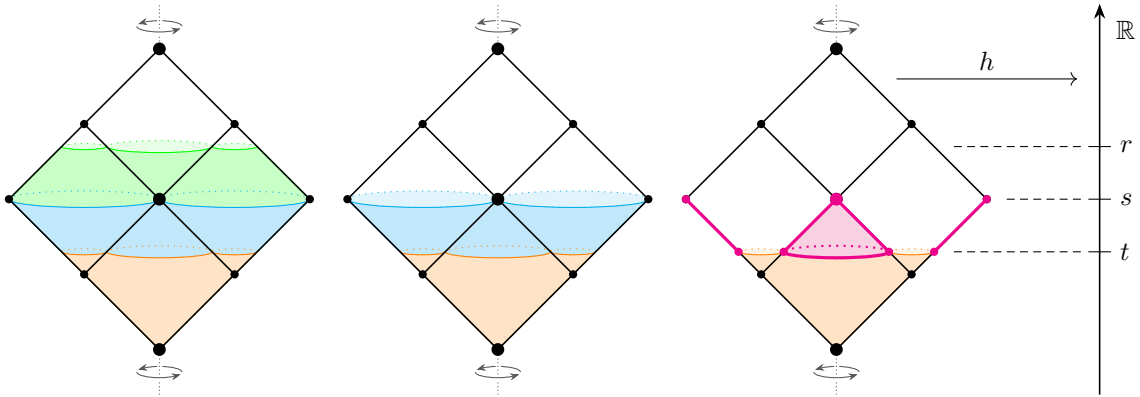
\[
\mspace{-60mu}
\begin{tikzpicture}
\begin{scope}[xshift=-9.0cm]
\morse[toplevel=true,  midlevel=true,  conedoff=false, dlinks=false, height=false]
\end{scope}
\begin{scope}[xshift=-4.5cm]
\morse[toplevel=false, midlevel=true,  conedoff=false, dlinks=false, height=false]
\end{scope}
\morse[toplevel=false, midlevel=false, conedoff=true,  dlinks=false]
\end{tikzpicture}
\mspace{-90mu}
\]
The first and the second pictures indicate that $Y^{h\le r}$ and~$Y^{h\le s}$ are $(\Z/2\Z)$-homotopy equivalent, since there are no vertices with height in~$(s,r]$ (case~(i) above).
The third picture illustrates~$Y^{h\le t}$ with cones on all the descending links of vertices at height~$s$.
The point is that this latter space is $(\Z/2\Z)$-homotopy equivalent to~$Y^{h\le s}$, the space depicted in the second picture (case~(ii)).
\end{example}

\begin{proof}[Proof of the \autoref{morse-lemma}]
We follow the proofs of~\cite{bestvina97}*{Lemmas~2.3 and~2.5}.
The main construction is a strong deformation retraction from~$Y^{h\le s}$ onto its subspace~$P$, and our goal is to show that it is $G$-equivariant at every stage.
The strong deformation retraction is constructed cell by cell.
For each~$r\in \R$ and each cell~$e$, with characteristic function $\MOR{\chi_e}{C_e}{e}$, set
\[
C_e^{h\le r}(\chi_e)=(h\circ \chi_e)^{-1}\bigl((-\infty,r]\bigr)
\,.
\]
Given a cell~$e$ meeting~$Y^{h\le s}$, there is a strong deformation retraction~$H_\tau^{\chi_e}$ of~$C_e^{h\le s}(\chi_e)$ onto
\[
C_e^{h\le t}(\chi_e)\cup \bigcup\SET{F}{F\text{ is a face of }C_e\text{ with }F\subseteq C_e^{h\le s}(\chi_e)}
\,.
\]
For example, if no vertices of~$e$ have $h$-value~$s$ then the latter space is just~$C_e^{h\le t}(\chi_e)$.
The construction of the~$H_\tau^{\chi_e}$ is done in the proofs of~\cite{bestvina97}*{Lemmas~2.3 and~2.5} (see also \cite{bestvina08}*{Propositions~2.4 and~2.7}), and for the sake of completeness we discuss the details here.
First note that since~$h\circ \chi_e$ is the restriction of an affine map $\R^{n_e}\TO\R$, the intersection of~$C_e$ with $(h\circ \chi_e)^{-1}\bigl([t,s]\bigr)$ is itself a convex polyhedron~$K$ in~$\R^{n_e}$.
Let~$T$ be the ``top'' face $(h\circ \chi_e)^{-1}(s)$ of~$K$ and let~$B$ be the union of $(h\circ \chi_e)^{-1}(t)$ with all the~$K\cap F$ for~$F$ a face of~$C_e$ contained in~$C_e^{h\le s}(\chi_e)$.
The goal now is to find a strong deformation retraction of~$K$ onto~$B$.
The only way~$T$ could be a vertex is if~$K=B$, so assume~$T$ is positive dimensional.
Now we can strong deformation retract~$K$ onto~$\overline{\partial K - T}$ by radially projecting away from some point of~$\R^{n_e}$ just above an interior point of~$T$ (where ``just above'' means in the sense of the affine function $\R^{n_e}\TO\R$ whose restriction is~$h\circ \chi_e$).
Next, by inducting on the dimension of~$e$, there is a strong deformation retraction of~$\overline{\partial K - T}$ onto~$B$.
Combining these we get a strong deformation retraction of~$K$ onto~$B$.
Finally, we extend this to the desired~$H_\tau^{\chi_e}$ by setting it equal to the identity on~$C_e- K$.
Bestvina and Brady point out that the~$H_\tau^{\chi_e}$ can be constructed to satisfy two naturality properties, namely:
\begin{enumerate}[label=(\arabic*)]
\item
\label{nat1}
For any affine homeomorphism $\MOR{\alpha}{C_e}{C_d}$, we have $H_\tau^{\chi_d\circ \alpha}=\alpha^{-1}\circ H_\tau^{\chi_d}\circ \alpha$.
\item
\label{nat2}
For~$e'$ a face of~$e$, the restriction of~$H_\tau^{\chi_e}$ to the corresponding face of~$C_e$ is~$H_\tau^{\chi_{e'}}$ (conjugated by an affine homeomorphism as appropriate).
\end{enumerate}
The construction of the~$H_\tau^{\chi_e}$ above makes clear that we can also assume they satisfy the following:
\begin{enumerate}[resume*]
\item
\label{nat3}
For any affine homeomorphism $\MOR{\alpha}{C_e}{C_d}$ such that $h\circ \chi_e = h\circ \chi_d\circ \alpha$, we have $H_\tau^{\chi_d}=\alpha\circ H_\tau^{\chi_e}\circ \alpha^{-1}$.
\end{enumerate}
This is because the construction only depends on~$h\circ\chi_e$, not specifically on~$\chi_e$.
Conjugating~$H_\tau^{\chi_e}$ by~$\chi_e$, we get a strong deformation retraction between the corresponding subspaces of~$e$.
Let us denote this by~$H_\tau^e$, so
\[
H_\tau^e= \chi_e\circ H_\tau^{\chi_e}\circ \chi_e^{-1}
.
\]
Note that~$H_\tau^e$ does not depend on the choice of admissible characteristic function for~$e$, thanks to naturality condition~\ref{nat1}.
These strong deformation retractions on each cell piece together, thanks to naturality condition~\ref{nat2}, into a strong deformation retraction of~$Y^{h\le s}$ onto
\(
P=Y^{h\le t}\cup \bigcup \SET{e}{e\text{ is a cell with }e\subseteq Y^{h\le s}}
\).

Now we need to show it is $G$-equivariant at every stage, for which it suffices to show that $H_\tau^{g\cdot e}=g\circ H_\tau^e\circ g^{-1}$ for all $e$, $g$, and~$\tau$.
Choose (as per \autoref{affineGCW}) an affine homeomorphism $\MOR{\alpha}{C_e}{C_{g\cdot e}}$ such that $g\circ \chi_e = \chi_{g\cdot e}\circ \alpha$.
Since~$h$ is $G$-equivariant we have $h\circ \chi_e = h\circ \chi_{g\cdot e}\circ \alpha$.
By naturality condition~\ref{nat3}, this implies that $H_\tau^{\chi_{g\cdot e}}=\alpha\circ H_\tau^{\chi_e}\circ \alpha^{-1}$.
We conclude that
\begin{align*}
H_\tau^{g\cdot e}
&= \chi_{g\cdot e}\circ H_\tau^{\chi_{g\cdot e}}\circ \chi_{g\cdot e}^{-1}
\\
&= \chi_{g\cdot e}\circ \alpha\circ H_\tau^{\chi_e}\circ \alpha^{-1}\circ \chi_{g\cdot e}^{-1}
\\
&= g\circ \chi_e\circ H_\tau^{\chi_e}\circ \chi_e^{-1}\circ g^{-1}
\\
&= g\circ H_\tau^e\circ g^{-1}
,
\end{align*}
as desired.
\end{proof}

We are now ready to prove \autoref{morse-useful}, which describes how the \autoref{morse-lemma} is usually applied in practice.

\begin{proof}[Proof of \autoref{morse-useful}]
First, assume $s< \infty$.
If $Y^{(0)}\cap h^{-1}\bigl((t,s]\bigr)=\emptyset$ then Lemma~\ref{morse-lemma} immediately gives the result, so assume $Y^{(0)}\cap h^{-1}\bigl((t,s]\bigr)$ is nonempty. Since $h(Y^{(0)})$ is closed and discrete, $h\bigl(Y^{(0)}\cap h^{-1}\bigl((t,s]\bigr)\bigr)$ is finite, so by induction and by the first case we can assume that $Y^{(0)}\cap h^{-1}\bigl((t,s)\bigr)=\emptyset$.
Now Lemma~\ref{morse-lemma} says that $Y^{h\le s}$ is $G$-homotopy equivalent to the pushout of~\eqref{cd:pushout}, or, equivalently, that diagram~\eqref{cd:commutative} is $G$-homotopy cocartesian.
So, to show that the right-hand vertical map $Y^{h\le t}\INTO Y^{h\le s}$ in~\eqref{cd:commutative} is a $G$-homotopy equivalence, it suffices to show that the same is true for the left-hand vertical map.
Since each $\dlk y$ and~$\dst y$ is $\Stab_G(y)$-contractible, each inclusion $\dlk y\INTO \dst y$ is a $\Stab_G(y)$-homotopy equivalence.
Using~\eqref{eq:induction-many}, we conclude that the left-hand vertical map in~\eqref{cd:commutative} is a $G$-homotopy equivalence, because induction sends $\Stab_G(y)$-homotopy equivalences to $G$-homotopy equivalences, and disjoint unions of $G$-spaces preserve $G$-homotopy equivalences.
This finishes the proof when~$s< \infty$.

Finally, assume~$s=\infty$.
Since $Y^{h\le s'}\INTO Y^{h\le s''}$ is a closed inclusion for each $s'\le s''\le \infty$, up to $G$-homotopy equivalence we have
\[
Y
=
\bigcup_{t<s'\in \N} Y^{h\le s'}
\simeq
\hocolim_{t<s'\in \N} Y^{h\le s'}
,
\]
because taking fixed points and homotopy groups commute with sequential colimits along closed inclusion.
By the previous case, $Y^{h\le t}\INTO Y^{h\le s'}$ is a $G$-homotopy equivalence for each  $t<s'<\infty$, and so the result follows.
\end{proof}


\section{Vietoris--Rips complexes}
\label{RIPS}

In this last section, we first define Vietoris--Rips complexes, and then we apply the equivariant discrete Morse theory from the previous section to prove our main result.
Our Morse-theoretic approach to the study of Vietoris--Rips complexes originates  in and is modeled after~\cite{zaremsky18}.

The definition of Vietoris--Rips complex that we adopt here is not the typical definition.
The typical definition yields a simplicial complex whose barycentric subdivision equals ours, and so the two definitions are naturally homeomorphic.
Our definition is formulated in terms of posets, and the language and tools from the world of posets turn out to be convenient for the arguments below, e.g., in the proofs of \autoref{Rips=E_G}(i) and \autoref{link-criterion}.
Let us start by reviewing geometric realizations of posets.

Given a poset~$\mathcal{P}$, its geometric realization~$\REAL{\mathcal{P}}$ is the simplicial complex obtained by geometric realization of the order complex of~$\mathcal{P}$, which is the abstract simplicial complex whose vertices are the elements and whose simplices are the proper chains of~$\mathcal{P}$.
There is another construction that produces a topological space naturally homeomorphic to~$\REAL{\mathcal{P}}$: think of the poset~$\mathcal{P}$ as a small category, take its nerve, which is a simplicial set, and then take the geometric realization of the nerve.

By a $G$-poset we mean a poset together with an action of a group~$G$ by poset maps, i.e., order-preserving functions.
The geometric realization of a $G$-poset is an affine $G$-CW complex.

We are now ready to introduce Vietoris--Rips complexes.
As in the previous section, we formulate our definitions in the presence of a group action.
When the action (or the group) is trivial one recovers the nonequivariant definitions.

\begin{definition}[Vietoris--Rips complexes]
Let $X$ be a set with an action of a group~$G$.
Let $\poset(X)$ be the poset of finite nonempty subsets of~$X$ ordered by inclusion, together with the evident induced $G$-action by poset maps.
Take its geometric realization to define $\rips_\infty(X)=\REAL{\poset(X)}$.

If $X$ is a metric space and $G$ acts by isometries, for any~$t\in \R$ let
\[\poset^{\le t}(X)=\SET{S\in \poset(X)}{\diam(S)\le t}.\]
Here $\MOR{\diam}{\poset(X)}{\R}$ denotes the diameter function, which is $G$-equivariant with respect to the trivial action on~$\R$, and so $\poset^{\le t}(X)$ is a $G$-subposet of~$\poset(X)$.
Take its geometric realization to define the \emph{Vietoris--Rips complex} $\rips_t(X)=\REAL{\poset^{\le t}(X)}$.
\end{definition}

Notice that we can identify~$\rips_t(X)$ with the full $G$-subcomplex of $\rips_\infty(X)$ spanned by those $S\in \poset(X)$ with $\diam(S)\le t$.

The link between Vietoris--Rips complexes and universal spaces for proper actions is provided by the following basic lemma.
In the special case when $X=G$, part~(i) already appears in~\cite{abels78}*{Example~2.6}.

\begin{lemma}
\label{Rips=E_G}
Let $X$ be a nonempty set with an action of a discrete group~$G$.
\begin{enumerate}
\item
If $\Stab_G(S)$ is finite for each $S\in \poset(X)$, then $\rips_\infty(X)=\Eu{G}$.
\item
If $X$ is a metric space and $G$ acts properly by isometries, then $\rips_t(X)$ is a proper $G$-CW complex for any $t>0$ or $t=\infty$.
\item
If $X$ is a metric space whose balls are all finite, and $G$ acts cocompactly by isometries, then $\rips_t(X)$ is a finite $G$-CW complex for any finite~$t>0$.
\end{enumerate}
\end{lemma}

\begin{proof}
(i)
The condition on the stabilizers implies that all vertices, and hence all points, of~$\rips_\infty(X)$ have finite stabilizers, and so~$\rips_\infty(X)$ is a proper $G$-CW complex.
It remains to show that $\rips_\infty(X)^H$ is contractible for any finite~$H\le G$.
This is proved by observing that $\rips_\infty(X)^H\cong \REAL{\poset(X)^H}$, i.e., taking fixed points and geometric realizations commute, and that the poset~$\poset(X)^H$ is directed and nonempty (since for example it contains the finite subset~$H\cdot x$ for any~$x\in X$).

(ii)
If $G$ acts properly, then the condition on the stabilizers in~(i) is satisfied, and so $\rips_t(X)$ is a proper $G$-CW complex.

(iii)
The conditions ensure that $\rips_t(X)$ is locally finite, and has finitely many $G$-orbits of vertices, so the action is cocompact.
\end{proof}

\autoref{Rips=E_G} offers a clear strategy to produce finite universal spaces for proper actions:
given a group~$G$ acting properly and cocompactly by isometries on a metric space~$X$ whose balls are all finite, try to use equivariant Morse theory to show that the finite $G$-CW complex~$\rips_t(X)$ is $G$-homotopy equivalent to~$\rips_\infty(X)=\Eu{G}$.

The key to applying Morse-theoretic tools to Vietoris--Rips complexes is the following function from~\cite{zaremsky18}.
Let $X$ be a metric space with an action of~$G$ by isometries.
Consider the function
\[
\MOR{(\diam,-\card)}{\poset(X)}{\R\times\R}
\,,
\quad
S\MAPSTO\bigl(\diam(S),-\card(S)\bigr)
,
\]
where $\diam(S)$ and~$\card(S)$ denote the diameter and the cardinality of~$S$, respectively.
Notice that this function is $G$-equivariant with respect to the trivial action on the target.
When every ball is finite and the action is cocompact, the function~$(\diam,-\card)$ yields an equivariant Morse function on~$\rips_\infty(X)$ as follows.

\begin{lemma}
\label{lexi}
Let $X$ be a metric space with a cocompact action of~$G$ by isometries.
Assume that every ball in $X$ is finite.
Then there exists a $G$-Morse function $\MOR{h}{\rips_\infty(X)}{\R}$ such that the filtration $\{\rips_t(X)\}_{t\in\R}$ of $\rips_\infty(X)$ is a cofinal subfiltration of $\{\rips_\infty(X)^{h\le t}\}_{t\in\R}$.
\end{lemma}

\begin{proof}
Since every ball is finite, the topology on $X$ induced by the metric is the discrete topology.
Thus by cocompactness, there exist finitely many $G$-orbits in~$X$, say $G\cdot x_1, \dotsc, G\cdot x_\ell$.
Now every value in~$\diam\bigl(\poset(X)\bigr)\subseteq \R$, the set of diameters of finite subsets of~$X$, equals the distance from one of the~$x_i$ to some point in~$X$, and since balls are finite this implies that $\diam\bigl(\poset(X)\bigr)$ is closed and discrete.
Next observe that every ball of radius~$r$ has cardinality equal to the ball of radius~$r$ centered at some~$x_i$, so for each~$r$ there is a uniform bound on the cardinality of all balls of radius~$r$.
In particular, for any fixed $a\in \diam\bigl(\poset(X)\bigr)$, there are only finitely many $(a,b)\in (\diam,-\card)\bigl(\poset(X)\bigr)$.
Since $\diam\bigl(\poset(X)\bigr)$ is closed and discrete, this implies that the image $(\diam,-\card)\bigl(\poset(X)\bigr)$ with the lexicographic order can be embedded in an order-preserving way as a closed, discrete subset of $\R$, say via some map
\(
\MOR{\iota}{(\diam,-\card)\bigl(\poset(X)\bigr)}{\R}
\).
Now simply define $h=\iota\circ(\diam,-\card)$ on vertices, and extend it affinely to the simplices of~$\rips_\infty(X)$.
The statement about filtrations is clear because $\rips_t(X)=\rips_\infty(X)^{h\le \iota(t,0)}$ for all~$t\in \diam\bigl(\poset(X)\bigr)$.
\end{proof}

Assume that every ball in $X$ is finite and that $X$ is $G$-cocompact, and consider the Morse function $\MOR{h}{\rips_\infty(X)}{\R}$ from \autoref{lexi}.
We now discuss descending links with respect to this Morse function.
Since geometric realizations of posets are flag complexes, we can view $\dlk S$ as a subcomplex of~$\rips_\infty(X)$ as in \autoref{dlk-in-simplicial-cx}, namely as the geometric realization of the subposet of~$\poset(X)$ consisting of all elements related to~$S$ and having strictly smaller $h$-value.
So the descending link of a vertex~$S$ in $\rips_\infty(X)$ is spanned by vertices of two types.
The first type is given by
\[
\dulkposet S=
\SET{S^\vee\in\poset(X)}{S^\vee\subseteq S\text{ and }\diam(S^\vee)<\diam(S)},
\]
and the second type is given by
\[
\ddlkposet S=
\SET{S^\wedge\in\poset(X)}{S\subsetneq S^\wedge\text{ and }\diam(S^\wedge)=\diam(S)}.
\]
The vertices of the first type span the \emph{descending down-link}~$\dulk S$ and those of the second type span the \emph{descending up-link}~$\ddlk S$.
Since $S^\vee\subseteq S^\wedge$ for all $S^\vee$ and~$S^\wedge$ as above, the descending link $\dlk S$ decomposes as the join
\[
\dlk S \cong \dulk S * \ddlk S
\,.
\]
Notice that there are $\Stab_G(S)$-equivariant homeomorphisms $\dulk S\cong \REAL{\dulkposet S}$ and $\ddlk S\cong \REAL{\ddlkposet S}$.

The next result provides a sufficient condition for showing that the descending links are equivariantly contractible, so that \autoref{morse-useful} applies.
The main assumption is essentially an equivariant version of the Link Criterion from~\cite{zaremsky18}, reinterpreted in our framework.

\begin{proposition}
\label{link-criterion}
Let $X$ be a metric space with a cocompact action of~$G$ by isometries, and let~$t>0$.
Assume that every ball in $X$ is finite.
Assume that there exists a $G$-map $\MOR{\zeta}{\poset(X)}{\poset(X)}$ such that the following conditions hold for all~$S\in \poset(X)$ with~$\diam(S)>t$:
\begin{enumerate}[label=(\arabic*)]
\item
\label{link-criterion-down}
if $\zeta(S)\subseteq S$ then $\diam(\zeta(S))<\diam(S)$, and for all $S^\vee\subseteq S$ with $\diam(S^\vee)<\diam(S)$ also $\diam(S^\vee\cup \zeta(S))<\diam(S)$;
\item
\label{link-criterion-up}
if $\zeta(S)\not\subseteq S$ then for all $S^\wedge\supseteq S$ with $\diam(S^\wedge)=\diam(S)$ also $\diam(S^\wedge\cup \zeta(S))=\diam(S)$.
\end{enumerate}
Then for any $S\in \poset(X)$ with $\diam(S)>t$ the descending link $\dlk S$ is $\Stab_G(S)$-contractible,
and $\rips_t(X)$ is $G$-homotopy equivalent to $\rips_\infty(X)$.
\end{proposition}

\begin{proof}
We only need to prove the first statement, since the second then follows from \autoref{morse-useful}.

Fix $S\in\poset(X)$ with $\diam(S)>t$.
Since $\zeta$ is $G$-equivariant, we have that $\Stab_G(S)\cdot \zeta(S)=\zeta(S)$.
We now treat the two cases separately.

\ref{link-criterion-down}
Assume that $\zeta(S)\subseteq S$.
Then $\diam(\zeta(S))<\diam(S)$, and for all $S^\vee\subseteq S$ with $\diam(S^\vee)<\diam(S)$ also $\diam(S^\vee\cup Z)<\diam(S)$.
This means that $\zeta(S)\in \dulkposet S$, and for all $S^\vee\in \dulkposet S$ also $S^\vee\cup \zeta(S)\in \dulkposet S$.
So we can define a map
\[
\MOR{f}{\dulkposet S}{\dulkposet S}
\,,
\quad
S^\vee\MAPSTO S^\vee\cup \zeta(S)
\,.
\]
Notice that $f$ is a $\Stab_G(S)$-equivariant poset map, and for all $S^\vee\in\dulkposet S$ we have
\[
S^\vee\subseteq f(S^\vee)\supseteq \zeta(S)
\,.
\]
These conditions ensure that the maps
$\REAL{\id_{\dulkposet S}}$ and~$\REAL{f}$
are $\Stab_G(S)$-homotopic, as are $\REAL{f}$ and~$\REAL{c}$ for $c$ the map sending all~$S^\vee$ to~$\zeta(S)$.
So $\id_{\REAL{\dulkposet S}}=\REAL{\id_{\dulkposet S}}$ is $\Stab_G(S)$-homotopic to a constant map, i.e., $\dulk S\cong \REAL{\dulkposet S}$ is $\Stab_G(S)$-contractible.

\ref{link-criterion-up}
Assume that $\zeta(S)\not\subseteq S$.
Then for all $S^\wedge\supseteq S$ with $\diam(S^\wedge)=\diam(S)$ also $\diam(S^\wedge\cup \zeta(S))=\diam(S)$.
This means that for all $S^\wedge\in \ddlkposet S$ also $S^\wedge\cup \zeta(S)\in \ddlkposet S$.
So we can define a $\Stab_G(S)$-equivariant poset map
\[
\MOR{f}{\ddlkposet S}{\ddlkposet S}
\,,
\quad
S^\wedge\MAPSTO S^\wedge\cup \zeta(S)
\,,
\]
and for all $S^\wedge\in\ddlkposet S$ we have
\[
S^\wedge\subseteq f(S^\wedge)\supseteq S\cup \zeta(S)
\,.
\]
Notice that $S\cup \zeta(S)$ is in $\ddlkposet S$, since $\zeta(S)\not\subseteq S$.
Then arguing as above we can conclude that $\ddlk S\cong \REAL{\ddlkposet S}$ is $\Stab_G(S)$-contractible.

Since $\dlk S$ is the join of $\dulk S$ and $\ddlk S$, in either case we conclude $\dlk S$ is $\Stab_G(S)$-contractible.
\end{proof}

We can now prove our main result.

\begin{theorem}
\label{main-precise}
Let $G$ be a group acting properly and cocompactly by isometries on an asymptotically CAT(0) geodesic metric space $X$.
Consider the orbit $G\cdot x_0$ of an arbitrary point $x_0\in X$, with the induced metric from~$X$.
Then, for any sufficiently large $t\in\R$, the Vietoris--Rips complex~$\rips_t(G\cdot x_0)$ is a finite model for the universal space for proper actions~$\Eu{G}$.
\end{theorem}

The strategy is of course to apply \autoref{link-criterion} to~$G\cdot x_0$.
The idea for the construction of the map
\(
\MOR{\zeta}{\poset(G\cdot x_0)}{\poset(G\cdot x_0)}
\)
comes from the following elementary observation:
in euclidean space, the diameter of a finite set~$S$ of points decreases if we replace~$S$ by the set of midpoints of pairs in~$S$, whereas if we add these midpoints to~$S$ the diameter does not change but the cardinality increases.
Some technical modifications are required to make this work for~$G\cdot x_0$, since first of all midpoints of points in an orbit need not lie in the orbit, and secondly the ambient space~$X$ is not euclidean but only asymptotically CAT(0), so extra care is needed for the diameter estimates.
Here are the details.

\begin{proof}
Let $X_0=G\cdot x_0$ for some arbitrary but fixed $x_0\in X$.
Since the action of $G$ on $X$ is proper, every ball in $X_0$ is finite.
Since $G$ acts properly and cocompactly by isometries on~$X$, all other assumptions of \autoref{Rips=E_G} are satisfied by~$X_0$, and so we conclude that $\rips_\infty(X_0)$ is a model for~$\Eu{G}$, and $\rips_t(X_0)$ is a finite proper $G$-CW complex.

Hence it is enough to show that $\rips_t(X_0)$ is $G$-homotopy equivalent to~$\rips_\infty(X_0)$ for any sufficiently large $t\in\R$.
In order to do this, we want to apply \autoref{link-criterion}, and so we need to construct a $G$-map~$\MOR{\zeta}{\poset(X_0)}{\poset(X_0)}$.

Since the action of~$G$ on~$X$ is cocompact, there exists an~$r>0$ such that every point in~$X$ is within~$r$ of some point of~$X_0$.
Let $\sigma$ be a sublinear function as in \autoref{asymp-CAT(0)}.
Now choose~$t$ large enough so that $t>\frac{2}{2-\sqrt{3}}r$ and for all $t'\ge t$ we have $\frac{\sigma(t')}{t'}<1-\frac{r}{t'}-\frac{\sqrt{3}}{2}$ (the first condition ensures $0<1-\frac{r}{t}-\frac{\sqrt{3}}{2}$, making the second condition possible).
Define
\[
\MOR{\omega}{\poset(X_0)}{\poset(X_0\times X_0)},
\quad
\omega(S)=\SET{(x,y)\in S\times S}{d(x,y)=\diam(S)},
\]
so $\omega(S)$ consists of all the pairs of points in~$S$ ``witnessing'' the diameter of~$S$.
With respect to the diagonal action on the target, the map $\omega$ is $G$-equivariant.
Now for each pair $(x,y)\in X_0\times X_0$ choose a point $\mu(x,y)\in X_0$ that is at distance at most~$r$ from the midpoint of some geodesic in~$X$ from $x$ to~$y$.
Since $G$ acts by isometries, we can do this in a way that leads to a $G$-equivariant function $\MOR{\mu}{X_0\times X_0}{X_0}$.
Finally define
\[
\MOR{\zeta}{\poset(X_0)}{\poset(X_0)},
\
\zeta(S)=\mu(\omega(S))=\SET{\mu(x,y)}{x,y\in S,\ d(x,y)=\diam(S)}
\]
and notice that $\zeta$ is $G$-equivariant by construction.

Now let $S\in \poset(X_0)$ with $\diam(S)>t$.
We claim that the conditions \ref{link-criterion-down} and~\ref{link-criterion-up} of \autoref{link-criterion} hold.

\ref{link-criterion-down}
If $\zeta(S)\subseteq S$, we need to show that $\diam(\zeta(S))<\diam(S)$, and that for all $S^\vee\subseteq S$ with $\diam(S^\vee)<\diam(S)$ also $\diam(S^\vee\cup \zeta(S))<\diam(S)$.
It suffices to show that for any $z\in \zeta(S)$ and $s\in S$, $d(z,s)<\diam(S)$.
Say $z=\mu(x,y)$ for $x,y\in S$, and let $\tilde{z}$ be the midpoint of a geodesic from $x$ to $y$ such that $d(z,\tilde{z})\le r$.
By the triangle inequality it now suffices to show that $d(\tilde{z},s)<\diam(S)-r$.
Using comparison triangles we see that $d(\tilde{z},s)\le \frac{\sqrt{3}}{2}\diam(S)+\sigma(\diam(S))$.
Since $\sigma(\diam(S))<\diam(S)\bigl(1-\frac{r}{\diam(S)}-\frac{\sqrt{3}}{2}\bigr)$, we get
\[
d(\tilde{z},s)< \frac{\sqrt{3}}{2}\diam(S) + \diam(S)\left(1-\frac{r}{\diam(S)}-\frac{\sqrt{3}}{2}\right) = \diam(S)-r
\]
as desired.

\ref{link-criterion-up}
If $\zeta(S)\not\subseteq S$, we need to show that for all $S^\wedge\supseteq S$ with $\diam(S^\wedge)=\diam(S)$ also $\diam(S^\wedge\cup\zeta(S))=\diam(S)$.
It suffices to show that for any $z\in\zeta(S)$ and $\hat{s}\in S^\wedge$, $d(z,\hat{s})\le \diam(S)$.
This follows by the same calculation as above, using $\hat{s}$ in place of $s$, since we have $d(\hat{s},x),d(\hat{s},y)\le\diam(S)$ and hence $d(\tilde{z},\hat{s})\le\frac{\sqrt{3}}{2}\diam(S)+\sigma(\diam(S))$.

Thus both conditions of \autoref{link-criterion} are satisfied, completing the proof.
\end{proof}

\begin{example}[Hyperbolic groups]
\label{hyperbolic}
Let $G$ be a $\delta$-hyperbolic group, which here we take to mean that geodesic triangles are $\delta$-thin as in~\cite{kar08}*{Section~2.2}.
Let $X$ be the Cayley graph of $G$.
Then $X$ is asymptotically CAT(0), and the sublinear function~$\sigma$ from \autoref{asymp-CAT(0)} can be taken to be constant~$\delta$~\cite{kar08}*{Proposition~3}.
Inspecting the proof of \autoref{main-precise}, noting that every point in~$X$ is within distance~$1/2$ of an element of~$G=G\cdot 1$, we see that $\rips_t(G)=\Eu{G}$ for any $t>\frac{2\delta+1}{2-\sqrt{3}}$.
This bound is roughly $t>7.465\delta + 3.733$.
The bound in \cite{meintrup02} is $t\ge 16\delta+8$, but it is difficult to compare this to our bound since we use different notions of $\delta$-hyperbolicity, and converting between the definitions inevitably changes the~$\delta$, so we do not claim our bound is better than Meintrup--Schick's.
\end{example}

\begin{question}
In the proof of \autoref{main-precise}, the asymptotically CAT(0) assumption comes into play primarily because in a euclidean equilateral triangle of diameter~$t$ the distance from a vertex to the midpoint of the opposite edge is~$\frac{\sqrt{3}}{2}t$, which is less than~$t$.
One could replace the~$\frac{\sqrt{3}}{2}$ with any value less than~$1$ though, and everything would still work, albeit with different bounds.
So, the question arises, is there a natural condition weaker than asymptotically CAT(0) for which the proof of \autoref{main-precise} still works?
It would also be very interesting to find a condition similar to CAT(0) or asymptotically CAT(0) that is invariant under quasi-isometries (cf.~\autoref{qi}) and for which the result of \autoref{main-precise} is true.
\end{question}

\begin{question}
\label{Rips(G)}
Keep the assumptions and notation of \autoref{main-precise}.
The metric on~$G\cdot x_0$ induced by the asymptotically CAT(0) metric on~$X$ is quasi-isometric to the word metric on~$G$, by the Schwarz--Milnor Lemma.
Recall that this does not imply that the word metric on~$G$ is asymptotically CAT(0), as explained in~\autoref{qi}.
Moreover, Vietoris--Rips complexes are not homotopy invariant under quasi-isometries, and therefore we cannot conclude from our arguments that~$\rips_t(G)$ itself is an~$\Eu{G}$.
A special case when this conclusion does hold is when we can take~$X$ to be the Cayley graph of~$G$, as in the hyperbolic case; see \autoref{hyperbolic}.
This leads to the following question:
is it true or false that, if $G$ is an asymptotically CAT(0) or even a CAT(0) group, then its own Vietoris--Rips complex $\rips_t(G)$ with respect to a word metric is an~$\Eu{G}$ for large enough~$t$?
\end{question}


\bigskip

\end{document}